\documentclass{amsart}

\usepackage{graphicx,color}

\makeatletter
\@addtoreset{equation}{section}

\makeatother

\newtheorem{theorem}{Theorem}[section]
\newtheorem{proposition}[theorem]{Proposition}

\newtheorem{lemma}[theorem]{Lemma}
\theoremstyle{definition}

\newtheorem{remark}[theorem]{Remark}

\newtheorem{example}[theorem]{Example}

\begin{document}

\title[Mean-convex flow developing infinitely many singular epochs]{An example of a mean-convex mean curvature flow developing infinitely many singular epochs}
\author{Tatsuya Miura}
\address{Graduate School of Mathematical Sciences, The University of Tokyo, 3-8-1 Komaba, Meguro, Tokyo, 153-8914 Japan}
\email{miura@ms.u-tokyo.ac.jp}

\begin{abstract}
In this paper, we give an example of a compact mean-convex hypersurface with a single singular point moved by mean curvature having a sequence of singular epochs (times) converging to zero.
\end{abstract}

\maketitle

\section{Introduction}

The regularity and singularity of mean curvature flow, which is a one-parameter family of hypersurfaces in $\mathbb{R}^{n+1}$ moving by its mean curvature, have been studied by many authors.
There is an excellent survey paper \cite{ColdingMinicozziPedersen15} on this issue from classical results to recent developments.

In particular it is well-studied for mean-convex flows, namely mean curvature flows of hypersurfaces with positive mean curvature.
A well-known conjecture about such flows is: any mean-convex flow from a smooth initial surface develops singularities only at finitely many epochs (for example see \cite{Wang11}).

The main result of this paper, Theorem \ref{maintheorem}, shows that there is a chance that the set of singular epochs is not finite even if an initial surface has only one singular point.
Such an example is rigorously given in the section 3 but rough shape of the initial surface is as drawn in Figure \ref{figgamma0}.

Our initial surface is constructed by dilation.
Thus it is self-similar near the singularity.
Using the self-similarity, we prove that the flow from the surface pinches at infinitely many epochs (times) $t_k\downarrow 0$ by comparing Angenent's doughnuts \cite{Angenent92} and balls.
One may be tempted to construct such a surface by using a rescaled periodic function.
However, this simple idea does not work directly, although idea of rescaling is important.
Our construction looks slightly complicated because we have to connect a ball like shape in a suitable way.
An advantage of our construction is that it is easy to confirm the desired properties like mean-convexity.
The feature of our construction is explained in detail in Remark \ref{remperiod}.

We describe the result in terms of the level set method introduced by Chen-Giga-Goto \cite{ChenGigaGoto91} and Evans-Spruck \cite{EvansSpruck91} (see a self-contained book by Giga \cite{Giga06} for details).
This method can define a {\it (generalized) interface evolution} of mean curvature flow for all times through singularities.
The interface evolution is uniquely determined by a given initial surface, although in general it is not necessary unique in the sense of ``surface evolution''.
The reason is that interface evolutions can fatten, namely, have an interior point at some time.
However, our example is now mean-convex in the sense of White \cite{White00} hence it does not fatten.

In the rest of this section, we mention some related known results.
Our example is useful to contrast known results.

For any smooth compact initial surface, there is a unique classical solution of mean curvature flow at least locally in time (see e.g.\ \cite{Bellettini13}\cite{Mantegazza11}).
However it must develop singularities in finite time and it is complicated generally.
The first non simple singularity is given by Grayson \cite{Grayson89} called ``neck-pinching'', which inspires the result of this paper.

On the other hand, mean curvature flow has a smoothing effect due to its parabolicity.
A remarkable well-known result by Ecker and Huisken \cite{EckerHuisken91} is that any uniform Lipschitz initial surface admits a classical solution of mean curvature flow locally in time.
This result is proved by establishing local interior regularity estimates.
Some other results are also known.
For example, Evans and Spruck  \cite{EvansSpruck92b} proved a local interior regularity result for a level set flow provided that it is given locally as the graph of a continuous function.
In addition, the recent works of Tonegawa and his co-authors \cite{KasaiTonegawa14}\cite{TakasaoTonegawa13}\cite{Tonegawa14} show the local existence of a classical solution for $C^1$ initial surfaces in terms of the Brakke flow (with a transport term).
Our example suggests that these smoothing effects are crucially based on that an initial surface is locally represented by a graph.

As mentioned above, mean-convex flows are well-studied compared with general flows.
There are many results about the size or nature of the singular set of mean-convex flows (See e.g.\ \cite{ColdingMinicozzi14}\cite{ColdingMinicozzi15}\cite{White00}\cite{White03}\cite{White13} or the subsection 2.3 in \cite{ColdingMinicozziPedersen15} for details).
In particular, if $n=2$, the mean-convex interface evolution is smooth for almost every time \cite{White00}.
Our example shows that the set of singular times can be an infinite set (in a finite time interval).

The case that an initial surface is given by rotating a graph is studied even better in \cite{AltschulerAngenentGiga95}.
In particular, an axisymmetric compact smooth initial surface develops singularities only at finitely many epochs.
Our example is axisymmetric thus also complements the above result.

Finally, we mention the case of curve shortening flow ($n=1$).
In this case, recently, Lauer proved that any finite length Jordan curve is smoothed out instantly \cite{Lauer13}.
This result is in marked contrast to our higher dimensional result ($n\geq2$).
In fact, our example is of finite area and the image of a continuous injection from the $n$-dimensional sphere, although the result as \cite{Lauer13} is not valid.

\begin{figure}[tb]
	\begin{center}
		\includegraphics[width=75mm]{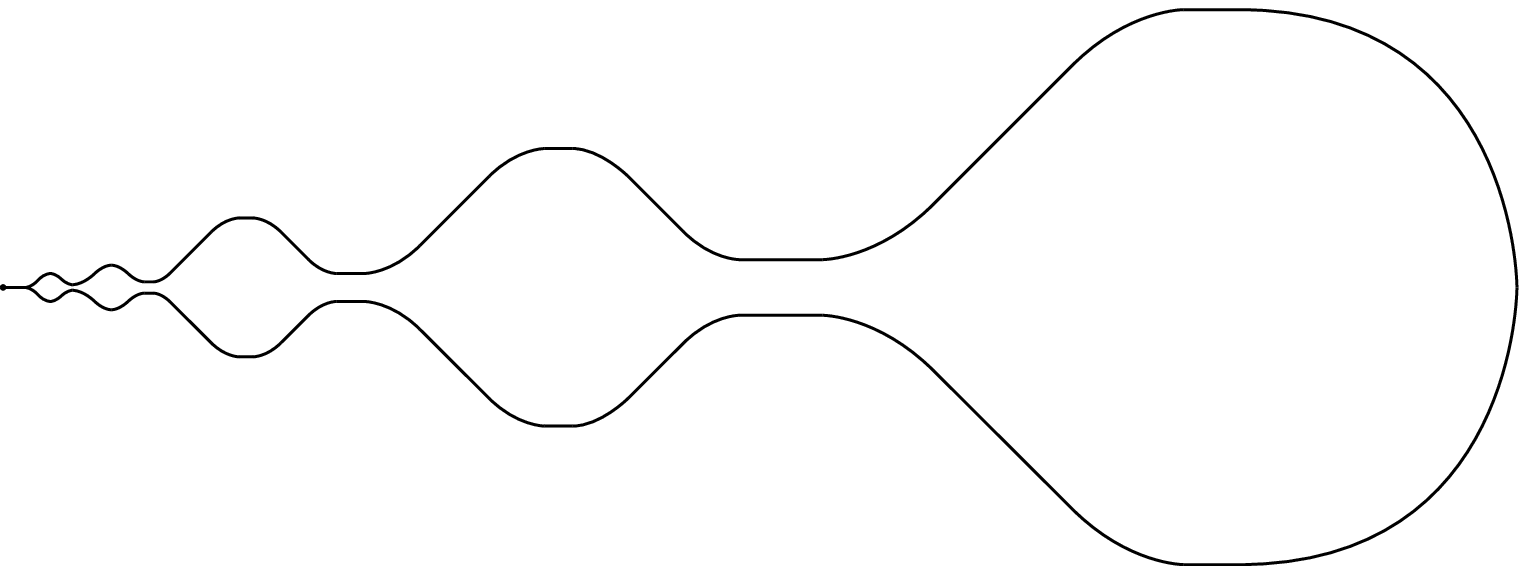}
	\end{center}
	\caption{example}
	\label{figgamma0}
\end{figure}

\section{An example of mean-convex hypersuface developing infinitly many singular epochs moved by mean curvature}

In this section we state our main theorem rigorously.
Throughout this paper, for a given open set $D_0$ (resp.\ boundary $\Gamma_0=\partial D_0$, closed set $E_0=D_0\cup\Gamma_0$) in $\mathbb{R}^{n+1}$, the set $D$ (resp.\ $\Gamma$, $E$) in $\mathbb{R}^{n+1}\times[0,\infty)$ denotes the open (resp.\ interface, closed) evolution of mean curvature flow and the set $D_t$ (resp.\ $\Gamma_t$, $E_t$) in $\mathbb{R}^{n+1}$ denotes its cross-section at time $t>0$.
See \cite{Giga06} for details of the above definitions.

Here is our main theorem.

\begin{theorem}\label{maintheorem}
Let $n \geq 2$.
There exists a compact connected axisymmetric initial hypersurface ${\Gamma}_0\subset\mathbb{R}^{n+1}$, which is the boundary of some bounded open set $D_0$ of finite perimeter, satisfying the following conditions:
\begin{itemize}
\item[{\rm(1)}] all points except one point in $\Gamma_0$ are $C^{\infty}$-regular and mean-convex points,
\item[{\rm(2)}] the generated evolutions satisfy the monotonicity $E_{t+h}\subset D_{t}$ for any $t\geq0$ and $h>0$, in particular $\Gamma_t$ does not fatten for any $t\geq0$,
\item[{\rm(3)}] for any $\tau>0$ there exists $0<t<\tau$ such that $\Gamma_t$ has a singularity.
\end{itemize}
\end{theorem}

A point ${\bf x} \in {\Gamma}_0$ is called {\it $C^{\infty}$-regular point} if there exists some open neighborhood $U$ in $\mathbb{R}^{n+1}$ containing ${\bf x}$ such that $U \cap {\Gamma}_0$ is an embedded $n$-dimensional $C^{\infty}$-manifold. A $C^{\infty}$-regular point ${\bf x} \in {\Gamma}_0=\partial D_0$ is called {\it mean-convex point} if the inward mean curvature at ${\bf x}$ is positive.

\begin{remark}
If we drop connectivity, an example of initial surface developing infinitely many singular epochs is easily provided by taking a countable union of dwindling spheres converging to a point.
\end{remark}

\begin{remark}
The monotonicity in (2) is the same to the mean-convexity of White \cite{White00}.
This monotonicity directly implies that the interface evolution does not fatten so that the level set flow is nothing but a Brakke flow (See also \cite{Giga06},\cite{Ilmanen94}).
\end{remark}

\section{Construction of an example}

We construct an example concretely in order to prove Theorem \ref{maintheorem}. 
This construction is based on the comparison principle of mean curvature flow (Lemma \ref{disjoint}) and two self-shrinking classical solutions (Example \ref{sphere} and \ref{doughnut}).
Using them, we can obtain a ``neck-pinching'' singularity as shown in \cite{Angenent92}.

\begin{lemma}[Avoiding property] \label{disjoint}
Let $\Gamma$, $\Gamma'\subset\mathbb{R}^{n+1}\times[0,\infty)$ be interface evolutions generated by compact initial surfaces $\Gamma_0$, $\Gamma_0'\subset\mathbb{R}^{n+1}$ respectively.
If $\Gamma_0$ and $\Gamma_0'$ are disjoint then so are $\Gamma$ and $\Gamma'$.
\end{lemma}

\begin{proof}
	See \cite[Theorem 4.5.2, Lemma 4.5.13]{Giga06}.
\end{proof}

\begin{example}[Spheres] \label{sphere}
The $n$-sphere $S^n\subset\mathbb{R}^{n+1}$ shrinks to its center without changing shape since the curvature is the same all around. 
The $n$-sphere with radius $R$ disappears at time $R^2/2n$.
\end{example}

\begin{example}[Shrinking doughnuts] \label{doughnut}
For $n\geq 2$, Angenent \cite{Angenent92} showed that there exists a self-shrinking doughnut $A^n\approx S^1\times S^{n-1}\subset\mathbb{R}^{n+1}$.
More precisely $A^n$ is created by rotating suitable simple closed curve $\gamma$ around the $x_0$-axis, where $\gamma$ lies in the $x_0x_1$-plane with $x_1>0$ and symmetric with respect to reflection in the $x_1$-axis.
The doughnut $A^n$ shrinks to its center without changing shape and disappears in finite time.
We define the radius of hole $r$ and the thickness $R$ of $A^n$ by
$$r:=\mathrm{min}\{x_1 \mid (x_0,x_1,0,\ldots,0)\in \gamma\},\quad R:=\mathrm{max}\{2x_0 \mid (x_0,x_1,0,\ldots,0)\in \gamma\}.$$
\end{example}

Now we construct our example.
Let $\phi_0:[0,\frac{1}{2}]\to[0,1]$ be a monotone increasing function of class $C^\infty$ such that $\phi_0\equiv0$ in $[0,\frac{1}{6}]$ and $\phi_0\equiv1$ in $[\frac{1}{3},\frac{1}{2}]$.
Fix a positive constant $\varepsilon_0\in(0,1)$ so that $(1+\max|\phi_0''|)\varepsilon_0^2<1$.
For $\delta\in(0,\frac{\varepsilon_0}{2})$ we define $f_\delta:\mathbb{R}\to[0,\varepsilon_0]$ by
\begin{align*}
f_\delta(x):=
\begin{cases}
(\varepsilon_0-\delta)\phi_0(x-\frac{3}{2})+\delta & (\frac{3}{2}<x\leq 2),\\
(\frac{\varepsilon_0}{2}-\delta)\phi_0(\frac{3}{2}-x)+\delta & (1<x\leq \frac{3}{2}),\\
0 & ({\rm otherwise}).
\end{cases}
\end{align*}

Next, let $\Omega_0$ be a planer convex domain in the $xy$-plane such that $\Omega_0$ is symmetric with respect to reflection in the $x$- and $y$-axis and its boundary $\partial\Omega_0$ is of class $C^\infty$, through four points $(0,\pm\varepsilon_0),(\pm1,0)\in\mathbb{R}^2$, straight in the region $\{|x|\leq\frac{1}{6}\}$ and has positive inner curvature at $(\pm1,0)\in\mathbb{R}^2$.
Then we define $\tilde{f}:\mathbb{R}\to[0,\varepsilon_0]$ so that the graph $y=\tilde{f}(x)$ with $x\in(0,1]$ is contained in $\partial\Omega_0$ and $\tilde{f}\equiv0$ elsewhere.

\begin{figure}[tb]
	\begin{center}
		\def\svgwidth{120mm}
		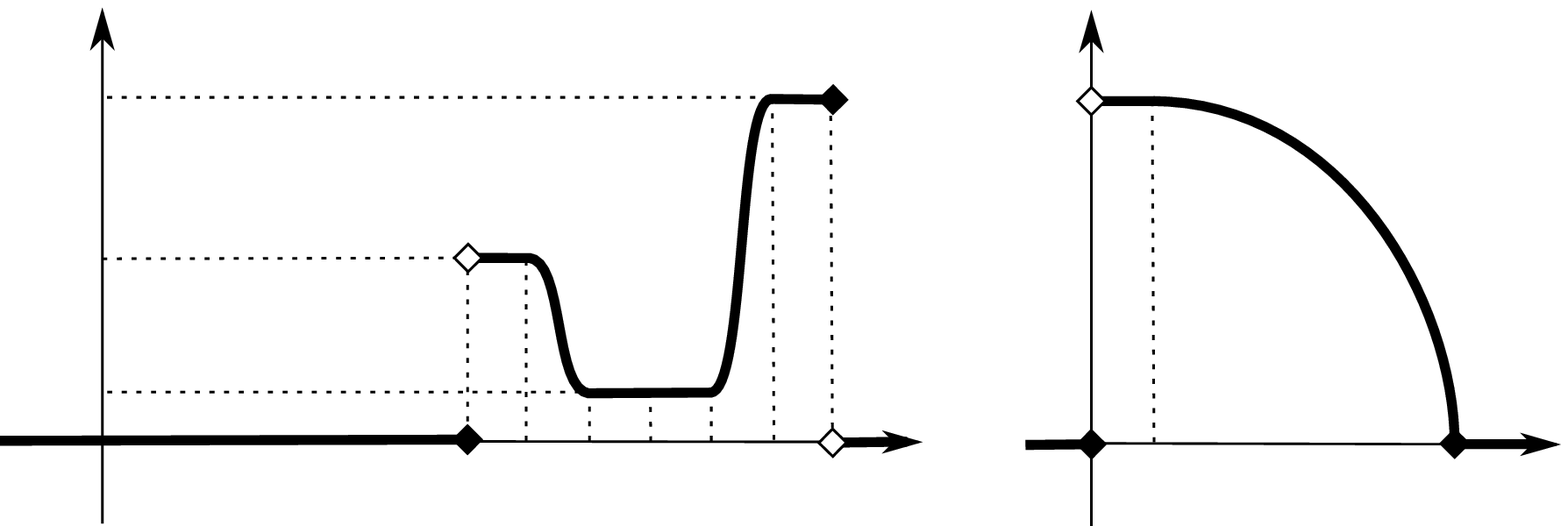
		\caption{graphs of $f_\delta$ and $\tilde{f}$}
		\label{figgraphfg}
	\end{center}
\end{figure}

Finally, we define $F:\mathbb{R}\rightarrow[0,\varepsilon_0]$ by
\begin{align*}
F(x):=\tilde{f}(x-2)+\sum_{k=0}^{\infty}2^{-k}f_{\delta_0}(2^kx),
\end{align*}
where $\delta_0\in(0,\frac{\varepsilon_0}{2})$ is taken sufficiently small so that there exists a self-shrinking doughnut $A^n_0$ with thickness $R_0<\frac{1}{3}$ and radius of hole $r_0>\delta_0$ such that $A^n_0$ disappears earlier than the $n$-sphere with radius $\frac{\varepsilon_0}{12}$.
Notice that $F$ is self-similar in $[0,2]$ in the sense that for $x\in[0,1]$
\begin{align}\label{selfsimilarity}
F(2x)=2F(x).
\end{align}

It turns out that the hypersurface $\widetilde{\Gamma}_0\subset\mathbb{R}^{n+1}$ created by rotating the graph of $F$ with respect to the $x_0$-axis (as Figure \ref{figgamma0}), namely $\widetilde{\Gamma}_0:=\partial\tilde{D}_0$ where
\begin{align*}\label{mainsurf}
\widetilde{D}_0:=\left\{(x_0,\ldots,x_n)\in\mathbb{R}^{n+1}\left|\ F(x_0)>\sqrt{{x_1}^2+\ldots+{x_n}^2}\right.\right\},
\end{align*}
satisfies all conditions of Theorem \ref{maintheorem}.
The surface $\widetilde{\Gamma}_0$ lies in the region $\{0\leq x_0\leq3\}$.
The origin is only one singular point in $\widetilde{\Gamma}_0$.

We shall check that this compact connected axisymmetric surface $\widetilde{\Gamma}_0$ of finite area satisfies the conditions of Theorem \ref{maintheorem}.
The following three propositions \ref{propof1}, \ref{propof2} and \ref{propof3} correspond to the three conditions (1), (2) and (3) respectively.

\begin{proposition}\label{propof1}
All points except the origin in $\widetilde{\Gamma}_0$ are $C^\infty$-regular mean-convex points.
\end{proposition}

\begin{proof}
It is easy to check that $\widetilde{\Gamma}_0$ is of class $C^\infty$ except the origin thus we only confirm the mean-convexity.
It suffices to confirm in the region $\{0<x_0\leq2\}$ since $\widetilde{\Gamma}_0$ is convex in $\{2<x_0\leq3\}$.
Moreover, the mean-convexity is preserved by dilation hence we only need to confirm in $\{1<x_0\leq2\}$.

The inward mean curvature of $\widetilde{\Gamma}_0$ in $\{0<x_0\leq2\}$ is represented as
\begin{align*}
\cfrac{n-1}{F(x_0)\sqrt{1+(F'(x_0))^2}}\ -\ \cfrac{F''(x_0)}{(\sqrt{1+(F'(x_0))^2})^3}.
\end{align*}
Therefore, to confirm its positivity it suffices to prove that for $1<x\leq2$ the inequality $F(x)F''(x)<1$ holds.
This inequality follows since in this case $F\equiv f_{\delta_0}$ holds and for any $\delta\in(0,\frac{\varepsilon_0}{2})$ we have $f_\delta\leq\varepsilon_0$ and
\begin{align*}
f_\delta''\leq (\varepsilon_0-\delta)|\phi_0''|+\delta \leq (1+\max|\phi_0''|)\varepsilon_0 < \varepsilon_0^{-1}.
\end{align*}
The last inequality follows from the definition of $\varepsilon_0\in(0,1)$.
\end{proof}

We denote the evolutions corresponding to $\widetilde{\Gamma}_0$ by $\widetilde{\Gamma}$, $\widetilde{D}$ and $\widetilde{E}$.

\begin{proposition}\label{propof2}
The monotonicity $\widetilde{E}_{t+h}\subset \widetilde{D}_{t}$ holds for any $t\geq0$ and $h>0$.
\end{proposition}

\begin{proof}
By order preserving property \cite[Theorem 4.5.2]{Giga06}, it suffices to prove that $\widetilde{E}_t\subset\widetilde{D}_0$ for any $t>0$.
For any positive integer $k$ we define $F_k:\mathbb{R}\rightarrow[0,\varepsilon_0]$ by
\begin{align*}
F_k(x):=
\begin{cases}
F(x) & (x>2^{-k+1}),\\
\varepsilon_0 2^{-k} & (2^{-k}\leq x\leq2^{-k+1}),\\
2^{-k}\tilde{f}(1-2^{k}x) & (x<2^{-k}),
\end{cases}
\end{align*}
and
$$\widetilde{D}^k_0:=\left\{(x_0,\ldots,x_n)\in\mathbb{R}^{n+1}\left|\ F_k(x_0)>\sqrt{{x_1}^2+\ldots+{x_n}^2}\right.\right\}$$
(see Figure \ref{figgammak}) and denote the corresponding closed evolution by $\widetilde{E}^k$.
By definition, we find that all points in $\partial\widetilde{D}^k_0$ are $C^{\infty}$-regular mean-convex points.
Since the classical mean-convexity implies the monotonicity \cite[Theorem 4.5.7]{Giga06}, we have $\widetilde{E}^k_t\subset\widetilde{D}^k_0$ for all $k$ and $t>0$.
Hence, noting the convergence $\widetilde{D}^k_0\downarrow\widetilde{D}_0$, we find that for all $t>0$
$$\bigcap_{k=0}^{\infty} \widetilde{E}^k_t\subset\bigcap_{k=0}^{\infty} \widetilde{D}^k_0=\widetilde{D}_0.$$
Using monotone convergence property \cite[Theorem 4.5.4]{Giga06}, we have $\widetilde{E}^k_t\downarrow\widetilde{E}_t$ for all $t>0$ thus we conclude that $\widetilde{E}_t\subset\widetilde{D}_0$ for any $t>0$.
\end{proof}

\begin{figure}[tb]
	\begin{center}
		\includegraphics[width=90mm]{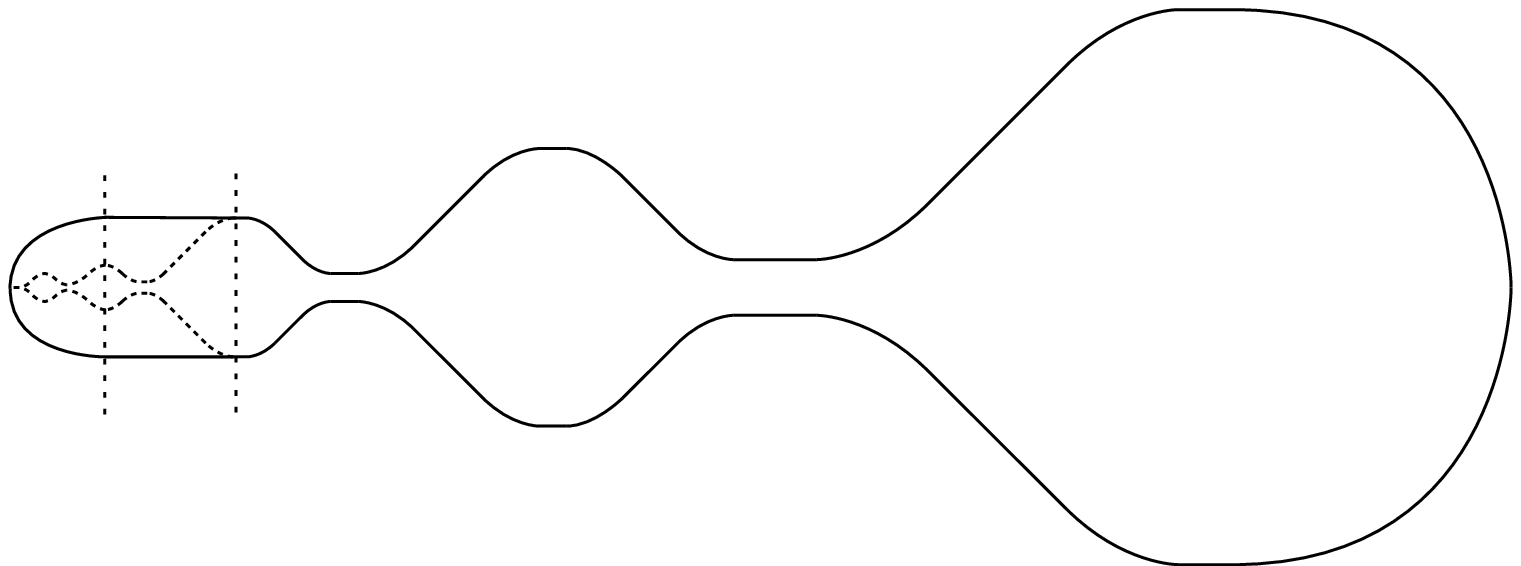}
	\end{center}
	\caption{$\partial\widetilde{D}^k_0$}
	\label{figgammak}
\end{figure}

\begin{proposition}\label{propof3}
For any $\tau>0$ there exists $0<t<\tau$ such that $\widetilde{\Gamma}_t$ has a singularity.
\end{proposition}

\begin{proof} 
Denote ${\bf e}_0:=(1,0,\dots,0)\in\mathbb{R}^{n+1}$.
Seeing our construction of $\widetilde{\Gamma}_0$, we notice that $\widetilde{\Gamma}_0$ encloses the two $n$-spheres with radius $\frac{\varepsilon_0}{12}$ centered at ${\bf e}_0$ and $2{\bf e}_0$.
Moreover, noting the definition of $\delta_0$, we also notice that $\widetilde{\Gamma}_0$ is circled by a self-shrinking doughnut centered at $\frac{3}{2}{\bf e}_0$ disappearing earlier than the spheres.
Then we find that the interface evolution $\widetilde{\Gamma}$ has a neck-pinching singularity at some time $t_0\in(0,\frac{\varepsilon_0^2}{144n})$.
By the self similarity (\ref{selfsimilarity}), for any positive integer $k$ we can take the two spheres with radius $\frac{\varepsilon_0}{12\cdot2^k}$ centered at $\frac{1}{2^{k}}{\bf e}_0$ and $\frac{1}{2^{k-1}}{\bf e}_0$ and the doughnut centered at $\frac{3}{2^{k+1}}{\bf e}_0$ disappearing earlier than the spheres as above.
We thus obtain a sequence of singular times $\{t_k\}$ of $\tilde{\Gamma}$ such that $t_k\in(0,\frac{\varepsilon_0^2}{144n\cdot4^k})$.
Since $t_k\downarrow0$, we obtain the consequence.
\end{proof}

\begin{remark}\label{remperiod}
As mentioned in the introduction, we are tempted to construct the self-similar part of an initial surface by a simpler scaling argument, for example rotating some rescaled periodic function as
$$f(x)=x\left(\phi(\varepsilon\log x)+\delta\right),$$
where $\phi$ is a suitable nonnegative periodic function and $\varepsilon$, $\delta$ are sufficiently small positive numbers.
This construction is simpler than ours and should provide a surface satisfying the main desired properties.
Unfortunately, we then need to be careful to confirm the properties rigorously.
For example in the proof of Proposition \ref{propof2}, we made a new surface by cutting and pasting smoothly.
In addition, the obtained surface should enclose the original one and remain mean-convex.
It is not trivial to confirm them for the simply constructed surface. 
However, the surface given in this paper is partially just straight and its overall shape is also clear so that there is no need to be careful in such a process.
\end{remark}

\section*{Acknowledgements}
The author would like to thank Yoshikazu Giga for helpful comments.
He is also grateful to anonymous referees for useful comments for improvement of this paper.
This work was supported by a Grant-in-Aid for JSPS Fellows 15J05166 and the Program for Leading Graduate Schools, MEXT, Japan.

\end{document}